\documentclass[reqno]{amsart}

\usepackage{amsmath,amssymb,amsthm,mathtools}
\usepackage[colorlinks=true]{hyperref}
\usepackage[capitalize,nameinlink]{cleveref}

\theoremstyle{plain}
\newtheorem{theorem}{Theorem}[section]
\newtheorem{lemma}[theorem]{Lemma}
\newtheorem{proposition}[theorem]{Proposition}
\newtheorem{corollary}[theorem]{Corollary}
\theoremstyle{definition}
\newtheorem{definition}[theorem]{Definition}

\theoremstyle{remark}

\numberwithin{equation}{section}

\newcommand{\F}{\mathbb{F}}
\newcommand{\Q}{\mathbb{Q}}
\newcommand{\Z}{\mathbb{Z}}
\newcommand{\CO}{\mathcal{O}}
\newcommand{\CU}{\mathcal{U}}
\newcommand{\CW}{\mathcal{W}}
\newcommand{\fp}{\mathfrak{p}}
\DeclareMathOperator{\bh}{BH}
\DeclareMathOperator{\ord}{ord}

\title[Generalized bent functions in the exceptional case]{Existence of generalized bent functions in the exceptional \texorpdfstring{$q\equiv2\pmod4$}{q = 2 mod 4}, odd-dimensional case}

\author{Jianing Li}
\address{Research Center for Mathematics and Interdisciplinary Sciences, Shandong University, Qingdao, Shandong 266237, China}
\email{lijn@sdu.edu.cn}

\author{Shi Ying}
\address{State Key Laboratory of Mathematical Sciences, Academy of Mathematics and Systems Science, Chinese Academy of Sciences, Beijing 100190, China}
\address{School of Mathematical Sciences, University of Chinese Academy of Sciences, Beijing 100049, China}
\email{yingshi@amss.ac.cn}

\author{Shenxing Zhang}
\address{School of Mathematics, Hefei University of Technology, Hefei, Anhui 230009, China}
\address{State Key Laboratory of Cyberspace Security Defense, Institute of Information Engineering, Chinese Academy of Sciences, Beijing 100085, China}
\email{zhangshenxing@hfut.edu.cn}


\begin{document}

\begin{abstract}
We resolve an open problem of Kumar, Scholtz, and Welch (1985) by constructing generalized bent functions from $(\mathbb{Z}/q\mathbb{Z})^m$ to $\mathbb{Z}/q\mathbb{Z}$ in the exceptional case $q\equiv2\pmod4$ with $m$ odd, the case their paper left without a construction and which four decades of subsequent work had addressed only through nonexistence results.
Concretely, for all integers $m\geq d\geq 2$ with $m=dr+2k$, $r\geq 1$, $k\geq 0$, we construct an explicit generalized bent function from $(\mathbb{Z}/q\mathbb{Z})^m$ to $\mathbb{Z}/q\mathbb{Z}$, where $q=2(2^d-1)$.
We further show that, when $d\ge3$, these generalized bent functions have Fourier coefficients that are not roots of unity --- all of them when $r$ is odd --- which gives a negative answer to a recent question of Armario, Egan, Kharaghani, and \'O~Cath\'ain about bent vectors for character tables.
\end{abstract}

\maketitle

\section{Introduction}\label{sec:introduction}

Bent functions are functions whose Fourier coefficients all have the smallest common magnitude allowed by Parseval's identity.
They were introduced for Boolean functions by Rothaus \cite{Rothaus1976} and extended to functions $f:(\Z/q\Z)^m\to\Z/q\Z$ by Kumar, Scholtz, and Welch \cite{KumarScholtzWelch1985}, which are called generalized bent functions of type $[m,q]$.
Besides their origins in sequence design and coding theory, generalized bent functions are equivalent to group-invariant Butson Hadamard matrices and are connected to difference sets
and finite geometry; see, for example, \cite{Schmidt2019,Tokareva2011}.

Kumar, Scholtz, and Welch constructed generalized bent functions of type $[m,q]$ with $m$ even or $q\not\equiv2\pmod4$, and raised as an open problem whether they exist in the remaining case
\begin{equation}\label{eq:exceptional-case}
  m\ \text{odd and}\ q\equiv2\pmod4.
\end{equation}
This problem has remained open ever since.
The evidence accumulated against existence: many partial nonexistence results were proved, ruling out one subfamily of~\eqref{eq:exceptional-case} after another---starting with Pei's \cite{Pei1993} for type $[1,14]$ and continuing in \cite{AkyildizGulogluIkeda1996, Ikeda1999, Feng2001, LiuMaFeng2002, FengLiu2003IEEE, FengLiu2003Acta, JiangDeng2015, LiDeng2017, LvLi2017, LeungSchmidt2019, LvZhu2025, YingDeng2026}---but no construction was found in this case.

The main contribution of this paper is to construct, for the first time, explicit generalized bent functions of the following type.

\begin{theorem}\label{thm:construction-intro}
  Let $d\geq 2,r\ge 1,k\ge 0$ be integers, $m=dr+2k$ and $q=2(2^d-1)$.
  Then there is an explicit generalized bent function
  \[
    f_{q,m}: (\Z/q\Z)^m\longrightarrow\Z/q\Z.
  \]
  Moreover, suppose $d\geq 3$.
  \begin{enumerate}
    \item If $r$ is odd, then no normalized Fourier coefficient $\CU_{f_{q,m}}(v)$ is a root of unity: all $q^m$ of them lie on the unit circle but outside the set of roots of unity.
    \item If $r$ is even, then a proportion $1-\binom{r}{r/2}2^{-r}$ of normalized Fourier coefficients $\CU_{f_{q,m}}(v)$ are not roots of unity.
  \end{enumerate}
  See \cref{thm:main2} for the precise statement.
\end{theorem}

This theorem provides a positive answer to the open problem about the existence of generalized bent functions mentioned above.

\begin{corollary}
  There exist generalized bent functions of exceptional types $[m,2(2^d-1)]$, where $m\geq d\geq 3$ are two odd integers.
\end{corollary}

In particular, generalized bent functions of types $[3,14]$ exist, which is exactly the smallest surviving parameter of the nonexistence literature.
Leung and Schmidt showed that for type $[3,2p^a]$ with $p$ an odd prime, existence would force $p=7$ \cite{LeungSchmidt2019}, leaving $[3,14]$ as the one open possibility in that family; \cref{thm:construction-intro}
realizes it.

Let $d\ge3$ and $n=2^d-1$.
Then clearly $\ord_n(2)=d$ and $-1$ is not a $2$-power in  $(\Z/n\Z)^\times$, equivalently, $2$ is not self-conjugate modulo $q=2n$.
When $n=2^d-1$ is prime, the numerical conditions of \cite{LeungSchmidt2019} are satisfied: $\ord_n(2)=d\le 2m-1$ and $n=2^d-1\le 2^{2m}+2^m+1$.
For composite $n$ the results of \cite{LeungSchmidt2019}, which assume $q=2p^a$ with $p$ an odd prime, do not apply.
In all cases the following classical obstruction is switched off: there is no generalized bent function of type $[m,2N]$ with $m$ odd when $2$ is self-conjugate modulo $N$ by \cite[Proposition~6]{KumarScholtzWelch1985} (see also \cite{Ikeda1999}).

The functions we construct have a further, unexpected feature.
Every generalized bent function produced by the classical constructions of \cite{KumarScholtzWelch1985}, and every one arising when $q$ is a prime power, has normalized Fourier coefficients that are roots of unity, by an argument of \cite{KumarScholtzWelch1985} going back to Kronecker.
Armario, Egan, Kharaghani and \'O Cath\'ain \cite[Question~3.8]{ArmarioEganKharaghaniOCathain2025} ask whether, for an abelian group $G$ of order $N$ and an $F(G)$-bent vector $\mathbf x$, the entries of $N^{-1/2}F(G)\mathbf x$ must be roots of unity.
Taking $G=(\Z/q\Z)^m$ and $\mathbf x=(\zeta_q^{f(g)})_{g\in G}$, these entries are precisely the normalized Fourier coefficients $\CU_f(v)$, so \cref{thm:construction-intro} with $d\ge3$ and $r=1$ gives a negative answer; the smallest example is $G=(\Z/14\Z)^{3}$.
In fact, in the exceptional case the opposite always holds as follows.

\begin{proposition}\label{prop:non-root-intro}
  Let $q\equiv2\pmod4$ and let $m$ be odd.
  If $f:(\Z/q\Z)^m\to\Z/q\Z$ is generalized bent, then no normalized Fourier coefficient $\CU_f(v)$ is a root of unity.
\end{proposition}

As noted in \cite{ArmarioEganKharaghaniOCathain2025}, a positive answer would, together with a result of
Craigen and Kharaghani, imply Ryser's circulant Hadamard conjecture; our counterexample is not cyclic, so the cyclic case, which is the one relevant to Ryser's conjecture, remains open.

Finally, we contrast our result with the existence of perfect nonlinearity functions.
A function $F:(\Z/q\Z)^m\to\Z/q\Z$ is \emph{perfect nonlinear} if $\eta\circ F$ is bent for every nonprincipal character $\eta$ of the codomain, a condition strictly stronger than generalized bentness.
When $q\equiv2\pmod4$ and $m$ is odd, no perfect nonlinear map $(\Z/q\Z)^m\to\Z/q\Z$ exists (see \cite{Rothaus1976, LogachevSalnikovYashchenko1997, CarletDing2004, Pott2004, Zhang2006, PoinsotPott2011}).

The paper is organized as follows.
In \cref{sec:definitions}, we recall the generalized bent condition, its Butson-matrix form, and prove the non-root-of-unity property.
In \cref{sec:construction}, we present the construction, compute all Fourier coefficients, and complete the proof of our main results.

\section{Generalized bent functions and Butson matrices}
\label{sec:definitions}

Throughout this paper, for any integer $h\geq 2$, we set $\zeta_h=e^{2\pi i/h}$.

\begin{definition}[Generalized bentness]\label{def:gbf}
  Let $q$, $m$ be positive integers with $q\geq 2$.
  A function $f:(\Z/q\Z)^m\to\Z/q\Z$ is called \emph{generalized bent} if
  \begin{equation}\label{eq:gbf-fourier}
    \Biggl\lvert
      \sum_{x\in(\Z/q\Z)^m}\zeta_q^{f(x)-v\cdot x}
    \Biggr\rvert^2=q^m
  \end{equation}
  for every $v\in(\Z/q\Z)^m$.
  This is the generalized bent function of type $[m,q]$ introduced in \cite{KumarScholtzWelch1985}.
\end{definition}

Write
\[
  \CW_f(v)=\sum_{x\in(\Z/q\Z)^m}\zeta_q^{f(x)-v\cdot x}, \qquad
  \CU_f(v)=q^{-m/2}\CW_f(v).
\]
Then $f$ is generalized bent precisely when every normalized
Fourier coefficient $\CU_f(v)$ lies on the
complex unit circle.

We also record the Butson-matrix interpretation of \cref{def:gbf}.
A \emph{Butson Hadamard matrix} $H\in\bh(N,h)$ is an $N\times N$ matrix whose entries are $h$-th roots of unity and which satisfies $HH^*=NI_N$.
If the rows and columns are indexed by a finite abelian group $A$, the matrix is \emph{$A$-invariant} when $H_{x,z}$ depends only on $x-z$.

\begin{proposition}[Group-invariant Butson correspondence]
  \label{prop:butson-correspondence}
  A function $f:(\Z/q\Z)^m\to\Z/q\Z$ is generalized bent if and only if
  \begin{equation}\label{eq:butson-matrix}
    H=\bigl(\zeta_q^{f(x-z)}\bigr)_{x,z\in(\Z/q\Z)^m}
  \end{equation}
  is a $(\Z/q\Z)^m$-invariant matrix in $\bh(q^m,q)$.
\end{proposition}

This is the character form of the standard group-ring criterion; see \cite[Lemma~2.1 and Proposition~2.3]{Schmidt2019}.

We close this section by proving \cref{prop:non-root-intro}, a basic property of the normalized spectrum in the exceptional case, independent of any explicit construction.
It contrasts with the prime-power case, where every $\CU_f(v)$ is a root of unity by an argument of \cite{KumarScholtzWelch1985} going back to Kronecker.

\begin{proof}[Proof of \cref{prop:non-root-intro}]
  Write $q=2n$ with $n$ odd, so $K=\Q(\zeta_q)=\Q(\zeta_n)$ has ring of integers $\CO_K$.
  Fix $v\in(\Z/q\Z)^m$, set $W=\CW_f(v)\in\CO_K$ and $\xi=\CU_f(v)=q^{-m/2}W$.
  Then  
  \[
    \xi^2=q^{-m}W^2\in K.
  \]
  Take a prime $\fp$ in $\CO_K$ lying above $2$.
  Since $q\equiv2\pmod 4$, the prime $2$ is unramified in $K$, so $v_{\fp}(q)=1$ and therefore $v_{\fp}(\xi^2)=2v_{\fp}(W)-m$ is odd.
  This forces that $\xi^2$ not to be a root of unity, and hence neither is $\xi$.
\end{proof}

For even $m$ the conclusion may fail.
The functions $g_{q,2k}$ defined in \cref{eq:g-q-2k} have root-of-unity coefficients.
So the non-root-of-unity property of $f_{q,m}$ for even $m$ is not automatic and carries independent information.

\section{Construction and proof}
\label{sec:construction}

In this section, let $m=dr+2k$ where $d\geq 2$, $r\geq 1$ and $k\ge 0$ are all integers.
Let $n=2^d-1$ and $q=2n$.

For $s=(s_1,\dots,s_d)\in \{0,1\}^d$, if $s\ne (0,\dots,0)$,  define
\begin{equation}\label{eq:kappa-def}
  \kappa(s)=\sum_{i=1}^d2^{i-1}s_i-1\in\{0,\dots,n-1\}.
\end{equation}
Clearly, the map
\begin{equation}\label{eq:kappa-bijection}
  \kappa:\F_2^d\setminus\{(0,\dots,0)\}\longrightarrow\Z/n\Z
\end{equation}
is a bijection by $n=2^d-1$.
We fix the Chinese remainder isomorphism
\begin{equation}\label{eq:crt}
  (\Z/q\Z)^d\stackrel{\sim}{\longrightarrow}\F_2^d\times(\Z/n\Z)^d,
  \qquad x\longmapsto(s,y), 
\end{equation}
together with the injection $\Z/n\Z\hookrightarrow\Z/q\Z$ given by multiplication by $2$.

Now we define the function $f_{q,d}$, which will be the core of the function $f_{q,m}$ of \cref{thm:construction-intro}.
\begin{definition}
  \label{def:block-function}
  For $(s,y)\in\F_2^d\times(\Z/n\Z)^d$, define
  \[
    b(s,y)=\begin{cases}
    \sum\limits_{i=1}^d y_i^2, &s=(0,\dots,0), \\[3mm]
    \sum\limits_{i=1}^{d-1}y_i^2+\kappa(s)y_d-\dfrac{\kappa(s)^2}{4}, &s\ne (0,\dots,0).
  \end{cases}
  \]
  Define $f_{q,d}:(\Z/q\Z)^d\to\Z/q\Z$ by
  \begin{equation}\label{eq:construction-f}
    f_{q,d}(x)=2b(s,y)\pmod{q},
  \end{equation}
  where $(s,y)$ is the image of $x$ under \cref{eq:crt}.
\end{definition}

Here $4^{-1}=2^{d-2}\in(\Z/n\Z)^\times$, so that $\kappa(s)^2/4$ in \cref{def:block-function} is unambiguous.

The Gauss sum will be useful in our proof.
Put 
\[
  \psi(t)=\zeta_n^t, \qquad
  G_n=\sum_{t\in\Z/n\Z}\psi(t^2).
\]
By \cite[Theorem~1.3.4]{BerndtEvansWilliams1998}, we have $G_n=i\sqrt{n}$ for any $n\equiv3\mod 4$.

\begin{lemma}
  For any $u\in\Z/n\Z$,
  \begin{equation}\label{eq:gauss-shift}
    \sum_{t\in\Z/n\Z}\psi(t^2-ut)
    =G_n\psi\Bigl(-\frac{u^2}{4}\Bigr).
  \end{equation}
\end{lemma}

\begin{proof}
  Conclude the equation by completing the square.
\end{proof}

\begin{theorem}\label{thm:minimal}
  For each $v\in (\Z/q\Z)^d$, $|\CU_{f_{q,d}}(v)|=1$.
  More precisely, $\CU_{f_{q,d}}(v)$ is equal to a $4n$-th root of unity times $\alpha^{\pm1}$ where $\alpha=2^{-d/2}(1-i\sqrt n)$.
  If moreover $d\geq 3$, then $\CU_{f_{q,d}}(v)$ is not a root of unity.
\end{theorem}

\begin{proof}
  We need to compute
  \[
    \CW_{f_{q,d}}(v)=q^{d/2} \CU_{f_{q,d}}(v)
  \]
  for each $v\in (\Z/q\Z)^d$.
  For $x,v \in (\Z/q\Z)^d$, write
  \begin{align*}
    s&=x\bmod2\in \F^d_2, &y&=x\bmod n\in (\Z/n\Z)^d,\\
    t&=v\bmod2\in \F^d_2, &w&=2^{d-1}v\bmod n\in (\Z/n\Z)^d.
  \end{align*}
  Since $\zeta_q^n=-1$ and $\zeta_q^{n+1}=\zeta_q^{2^d}$,
  \[
    \zeta_q^{-1}
    =-\zeta_q^{-2^d}
    =-\zeta_n^{-2^{d-1}}=(-1)\psi(-2^{d-1}),
  \]
  it follows that
  \[
    \CW_{f_{q,d}}(v)=\sum_{x\in(\Z/q\Z)^d}\zeta_q^{f_{q,d}(x)-v\cdot x}
    =\sum_{x\in(\Z/q\Z)^d}\zeta_n^{b(s,y)}\zeta_q^{-v\cdot x}=S(t,w),
  \]
  where
  \[
    S(t,w):=\sum_{s\in\F_2^d}\sum_{y\in(\Z/n\Z)^d}
    (-1)^{t\cdot s}\psi\bigl(b(s,y)-w\cdot y\bigr).
  \]
  Write
  \[
    A_s(w) = \sum_{y\in(\Z/n\Z)^d} \psi\bigl(b(s,y)-w\cdot y\bigr),
  \]
  so $S(t,w)= \sum_{s\in \F^d_2} (-1)^{t\cdot s}A_s(w)$.
  We will compute $A_s(w)$ for each $s \in\F^d_2$. 

  By \cref{eq:gauss-shift}, when $s=0$, we have 
  \[
    A_0(w)=G_n^d\Phi(w),
    \qquad \text{ where }
    \Phi(w)=\psi\left(-\frac{w_1^2+\dots+w_d^2}{4}\right);
  \]
  when $s\ne0$, we have
  \[
    A_s(w)=G_n^{d-1}
    \psi\left(-\frac{w_1^2+\dots+w_{d-1}^2+\kappa(s)^2}{4}\right)
    \sum_{y_d\in\Z/n\Z}\psi\bigl((\kappa(s)-w_d)y_d\bigr).
  \]
  The last sum is zero unless $\kappa(s)=w_d$ or equivalently $s=s_w:=\kappa^{-1}(w_d)$, in which case it is $n$. 
  Thus we have
  \begin{align*}
    S(t,w)&
    =\sum_{s\in \F^d_2}(-1)^{t\cdot s}A_s(w)
    =A_0(w)+(-1)^{t\cdot s_w}A_{s_w}(w)\\&
    =\Phi(w)\bigl(G_n^d+(-1)^{t\cdot s_w}nG_n^{d-1}\bigr)\\&
    =\Phi(w)i^dn^{d/2}\bigl(1-(-1)^{t\cdot s_w}i\sqrt n\bigr)
    =\Phi(w)i^dq^{d/2}\alpha^{(-1)^{t\cdot s_w}},
  \end{align*}
  where $G_n=i\sqrt{n}$.
  Therefore, we obtain that for each $v\in (\Z/q\Z)^d$, 
  \begin{equation}\label{eq:exact-coefficient}
    \CU_{f_{q,d}}(v)
    =q^{-d/2}\CW_{f_{q,d}}(v)
    =q^{-d/2}S(t,w)
    =\Phi(w)i^d\alpha^{(-1)^{t\cdot s_w}}.
  \end{equation}
  This clearly has absolute value $1$ since $\Phi(w)$ is an $n$-th root of unity whence $|\Phi(w)|=1$ and $n+1=2^d$. This proves the bentness of $f_{q,d}$.

  Assume that $d\geq 3$.
  Since $n=2^d-1\equiv 3\pmod 4$, $n$ is not a square, so $\alpha^2=2^{-d}(2-2^d-2i\sqrt n)$ is irrational.
  As $|\alpha^2|=1$, its minimal polynomial over $\Q$ is $X^2+(2-2^{2-d})X+1$, which is not integral for $d\geq 3$.
  Hence $\alpha^2$, and therefore $\alpha$, is not an
  algebraic integer, so $\alpha$ is not a root of unity.
  (For $d=2$ the conclusion genuinely fails: $\alpha=e^{-i\pi/3}$.)
\end{proof}

Denote $g_{q,2k}:(\Z/q\Z)^{2k}\to \Z/q\Z$ by
\begin{equation}\label{eq:g-q-2k}
  g_{q,2k}(x)=\sum_{j=1}^k x_jx_{k+j}.
\end{equation}
Then
\begin{equation}\label{eq:CU-g}
  \CU_{g_{q,2k}}(v)
  =\zeta_q^{-\sum_{j=1}^k v_jv_{k+j}}
\end{equation}
is a $4n$-th root of unity, and $g_{q,2k}$ is a bent function by \cite[Theorem~1]{KumarScholtzWelch1985}.

For any $m=dr+2k$, define $f_{q,m}:(\Z/q\Z)^m\to \Z/q\Z$ (depending on the choice of $r$) by
\[
  f_{q,m}(x)=\sum_{j=1}^r f_{q,d}(x^{(j)})+g_{q,2k}(x').
\]
where $x^{(j)}:=(x_{d(j-1)+1},\dots,x_{dj})$ for $1\le j\le r$ and
$x'=(x_{dr+1},\dots,x_m)$.
Here $g_{q,2k}$ disappears if $k=0$.
We shall prove the following result, which clearly implies \cref{thm:construction-intro}.

\begin{theorem}\label{thm:main2}
  For each $v\in (\Z/q\Z)^m$, $|\CU_{f_{q,m}}(v)|=1$.
  More precisely, set $\alpha=2^{-d/2}(1-i\sqrt n)$ and let
  \[
    a(v)=\#\bigl\{1\le j\le r : t^{(j)}\cdot s_{w^{(j)}}\ \text{is even}\bigr\}\in\{0,1,\dots,r\}
  \]
  where
  \[
    t^{(j)}=v^{(j)}\bmod 2
    \quad\text{and}\quad
    s_{w^{(j)}}=\kappa^{-1}(w_{dj})=\kappa^{-1}(2^{d-1}v_{dj}\bmod n)
  \]
  are notations in the proof of \cref{thm:minimal}.
  Then $\CU_{f_{q,m}}(v)$ equals a $4n$-th root of unity times $\alpha^{2a(v)-r}$, and $a(v)=a$ for exactly $\binom ra2^{-r}q^{m}$ of the $v\in(\Z/q\Z)^{m}$.

  If $d\ge3$, then $\CU_{f_{q,m}}(v)$ is a root of unity if and only if $r$ is even and $a(v)=r/2$; in particular, if $r$ is odd then no $\CU_{f_{q,m}}(v)$ is a root of unity.
\end{theorem}

\begin{proof}
  By \cite[Property~3]{KumarScholtzWelch1985}, $f_{q,m}$ is a generalized bent function.
  Moreover,
  \[
    \CU_{f_{q,m}}(v)
    =\prod_{j=1}^{r}\CU_{f_{q,d}}(v^{(j)})\CU_{g_{q,2k}}(v')
  \]
  the last factor being absent when $k=0$.
  By \cref{eq:exact-coefficient}, we have
  \[
    \CU_{f_{q,m}}(v)
    =\CU_{g_{q,2k}}(v')\prod_{j=1}^{r}\Phi(w^{(j)})i^d\alpha^{(-1)^{t^{(j)}\cdot s_{w^{(j)}}}}
    =\eta\alpha^{\sum_{j=1}^r(-1)^{t^{(j)}\cdot s_{w^{(j)}}}},
  \]
  where
  \[
    \eta=\CU_{g_{q,2k}}(v')i^{dr}\psi\Bigl(-\frac{w_1^2+\cdots+w_{dr}^2}4\Bigr)
  \]
  is a $4n$-th root of unity by \cref{eq:CU-g}.
  By the definition of $a=a(v)$, we have $\CU_{f_{q,m}}(v)=\eta \alpha^{2a-r}$.
  For a fixed $w^{(j)}$, we have $s_{w^{(j)}}\ne(0,\dots,0)$, so exactly half of the $t^{(j)}\in\F_2^{d}$ give each parity; the $r$ parities are therefore independent and equidistributed, which gives the stated frequency of each value of $a$.

  Finally let $d\ge3$.
  If $2a-r=0$, then $\CU_{f_{q,m}}(v)=\eta$ is a $4n$-th root of unity.
  If $2a-r\neq0$ and $\alpha^{2a-r}$ were a root of unity, then $\alpha$ itself would be a root of unity, contradicting \cref{thm:minimal}.
\end{proof}

\begin{corollary}
  Let $d\geq 2$ and $q=2(2^d-1)$.
  The matrix
  \[
    H=\bigl(\zeta_q^{f_{q,d}(x-z)}\bigr)_{x,z\in(\Z/q\Z)^d}
  \]
  is a $(\Z/q\Z)^d$-invariant matrix in $\bh(q^d,q)$. Its entries are in fact $n$-th roots of unity.
\end{corollary}

\begin{proof}
  The Butson property follows from \cref{prop:butson-correspondence}.
  Since $f_{q,d} \equiv 2b(s,y) \bmod {q}$ by definition, $\zeta_q^{f_{q,d}(x)}=\zeta_n^{b(s,y)}$ is an $n$-th root of unity.
\end{proof}

\section*{Acknowledgements}

J. Li is partially supported by Shandong Provincial Natural Science Foundation (ZR2025MS31) and by the Fundamental
Research Funds for the Central Universities.
S. Ying is partially supported by National Key R\&D Program of China (No. 2025YFA1017203).
S. Zhang is partially supported by State Key Laboratory of Cyberspace Security Defense (Grant No. 2025-MS-04).

\section*{Declaration of generative AI and AI-assisted technologies in the manuscript preparation process}

Some of the mathematical work presented in this paper was carried out with assistance from Eureka and subsequently verified by the authors.
Eureka is a multi-agent system developed by JIUCHONG at the University of Science and Technology of China for mathematical research through human--AI interaction.
The authors assume full responsibility for the content of this paper.

\end{document}